\UseRawInputEncoding 
\documentclass[12pt, twoside]{article}

\usepackage{amsmath,amsthm,amssymb}
\usepackage{times}
\usepackage{enumerate}
\usepackage{empheq}
\usepackage{url}
\usepackage{lineno}
\usepackage{microtype}
\usepackage{cases}
\usepackage[american]{babel}
\usepackage{indentfirst}
\def\titlerunning#1{\gdef\titrun{#1}}
\makeatletter
\def\author#1{\gdef\autrun{\def\and{\unskip, }#1}\gdef\@author{#1}}
\def\address#1{{\def\and{\\\hspace*{18pt}}\renewcommand{\thefootnote}{}%
\footnote {#1}}%
\markboth{\autrun}{\titrun}}
\makeatother
\def\email#1{e-mail: #1}
\def\subjclass#1{{\renewcommand{\thefootnote}{}%
\footnote{\emph{Mathematics Subject Classification (2020):} #1}}}
\def\keywords#1{\par\medskip
\noindent\textbf{Keywords.} #1}

\newtheorem{thm}{Theorem}
\newtheorem{prop}{Proposition}
\newtheorem{lem}{Lemma}
\newtheorem{de}{Definition}
\newtheorem{re}{Remark}
\numberwithin{equation}{section}

\DeclareMathOperator*{\supp}{supp}

\newcommand{\g}{\gamma}
\newcommand{\R}{\mathbb{R}}
\newcommand{\T}{\mathbb{T}}

\begin{document}
\baselineskip=15pt

\titlerunning{First Order Discounted Mean Field Games}
\title{Existence of solutions to contact mean field games of first order}
\author{Xiaotian Hu \and Kaizhi Wang}

\maketitle

\address{Xiaotian Hu: School of Mathematical Sciences, Shanghai Jiao Tong University, Shanghai 200240, China;
\email{sjtumathhxt@sjtu.edu.cn}
\and 
Kaizhi Wang (Corresponding author): School of Mathematical Sciences, Shanghai Jiao Tong University, Shanghai 200240, China;
\email{kzwang@sjtu.edu.cn}
}

\begin{abstract}
This paper deals with the existence of solutions of a class of contact mean field games systems of first order. 
Cardaliaguet \cite{CAR} found a link between the weak KAM theory for Hamiltonian systems and mean field games systems.  We prove that there is still a connection between the weak KAM theory for contact Hamiltonian systems and contact mean field games systems. By the analysis of properties of the Mather set for contact Hamiltonian systems, we prove the main existence result.

\end{abstract}



\section{Introduction}
\setcounter{equation}{0}
The mean field games system was introduced by Lasry and Lions \cite{bib:LL1,bib:LL2,bib:LL3} and Caines, Huang and Malham\'e \cite{bib:HCM1,bib:HCM2}. In this paper we only discuss first-order mean field games systems. It is a coupled system of partial differential equations, one Hamilton-Jacobi equation and one continuity equation. From the view of control theory, a Hamilton-Jacobi equation with an external mean field term is standard. The mean field term involves a probability distribution governed by a continuity  equation, which depends on the feedback and on the viscosity solution \cite{CL} of the Hamilton-Jacobi equation. The idea of equilibrium states in the mean field games theory, which are distributed along optimal trajectories generated from the feedback strategy, is quite  enlightening.

In some situations, the ergodic mean field games system
\begin{equation}\label{lab1}
	\begin{cases}
		\  K(x, Du)=F(x, m)+c(m) & \text{in} \quad X, \\ \ \text{div}\Big(m \frac{\partial K}{\partial p}(x, Du)\Big)=0  & \text{in} \quad X,  \\ \int_{X}{m\ dx}=1
	\end{cases}
\end{equation}  
can be described as the limit system of a finite time ($T>0$) horizon mean field games system
\begin{equation*}\label{lab2}
	\begin{cases}
		\ -\partial _{t} u^{T} + K(x, Du^{T})=F(x, m^{T}(t)) & \text{in} \quad (0,T)\times X, \\ \  \partial _{t}m^{T}-\text{div}\Big(m^{T}\frac{\partial K}{\partial p}(x, Du^{T})\Big)=0  & \text{in} \quad (0,T)\times X,  \\ \ m^{T}(0)=m_{0}, \quad u^{T}(T,x)=u^{f}(x), & x\in X
	\end{cases}
\end{equation*}
as $T$ goes to infinity.  
See \cite{CAR,CCMW,CCMW1} for this kind of results, where the state space $X$ is $\T^n=\R^n/\mathbb{Z}^n$, $\R^n$ and $\Omega\subset\R^n$, respectively. 
Obviously, ensuring the existence of solutions to the ergodic mean field games system is an important issue. 
Cardaliaguet \cite{CAR} discovered the link between the weak KAM theory and mean field games system \eqref{lab1}, and got an existence result for solutions to \eqref{lab1} for $X=\T^n$.

Let us see how to get a solution of \eqref{lab1} from the weak KAM point of view. In this paper we always use $M$ to denote 
a connected, closed (compact, without boundary) and  smooth manifold endowed with a Riemannian metric. A simple example is $M=\T^n$. Denote by $\mathrm{diam}(M)$ the diameter of $M$. We will denote by $(x,v)$ a point of the tangent bundle $T M$ with $x \in M$ and $v$ a vector tangent at $x .$ The projection $\pi: T M \rightarrow M$ is  $(x, v) \to x .$ The notation $(x, p)$ will designate a point of the cotangent bundle $T^{*} M$ with $p \in T_{x}^{*} M .$  Consider a Hamiltonian $K=K(x,p): T^*M \to \mathbb{R}$ which is $C^{2}$, superlinear and strictly convex in $p$.
Such a Hamiltonian is called a Tonelli Hamiltonian.
We can associate with $K$ a Lagrangian, as a function on $TM: l(x,v)=\sup _{p \in T^*_{x} M}\left\{\langle p, v\rangle_{x}-K(x,p)\right\}$, where $\langle\cdot, \cdot\rangle_{x}$ represents the canonical pairing between
the tangent and cotangent space. Sometimes, we use $p\cdot v$ to denote $\langle p, v\rangle_{x}$ for simplicity.

Let $\mu$ be a Mather measure \cite{Mat91} for the Euler-Lagrange equation 
\begin{align}\label{K}
	\frac{d}{dt}\frac{\partial l}{\partial v}=\frac{\partial l}{\partial x}.
\end{align}
Then $\mu$ is a closed measure (see, for instance, \cite{Ber}), i.e., 
\[
\int_{TM}v\cdot D\varphi(x)d\mu=0,\quad \forall \varphi\in C^1(M)，
\] 
where $C^k(M)$ $(k\in\mathbb{N})$ stands for the function space of continuously differentiable functions on $M$.
Recall that 
\[
\supp(\mu)\subset \tilde{\mathcal{A}}=\bigcap_{(u_-,u_+)}\{(x,v)\in TM\ : \ u_-(x)=u_+(x), Du_-(x)=Du_+(x)=\frac{\partial l}{\partial v}(x,v)\},
\]
where $\tilde{\mathcal{A}}$ is the Aubry set for Lagrangian system \eqref{K}, and the intersection is taken on the pairs $(u_-,u_+)$ of conjugate functions, i.e., $u_-$ (resp. $u_+$) is a backward (resp. forward) weak KAM solution of 
\begin{align}\label{lab3}
K(x,Du)=c(K)
\end{align}
and $u_-=u_+$ on the projected Mather set $\mathcal{M}$ of  system \eqref{K}. Here, $\mathcal{M}:=\pi \tilde{\mathcal{M}}$, where  $\tilde{\mathcal{M}}$ is the union of supports of Euler-Lagrange flow $\Phi_t^l$-invariant probability measures supported in $\tilde{\mathcal{A}}$, called the Mather set.
The symbol $c(K)$ denotes the Ma\~n\'e critical value of $K$. See Section 2 for the definition and representation formulas of Ma\~n\'e's critical value.
 Let $u_-$ be an arbitrary backward weak KAM solution (or equivalently \cite{Fat-b}, viscosity solution) of equation  \eqref{lab3}, and let $\sigma:=\pi\sharp\mu$, where 
$\pi\sharp\mu$ denotes the push-forward of $\mu$ through $\pi$. Then 
\[
0=\int_{\supp(\mu)}v\cdot D\varphi(x)d\mu=\int_{\supp(\sigma)}\frac{\partial K}{\partial p}(x,Du_-(x))\cdot D\varphi(x)d\sigma=\int_{M}\frac{\partial K}{\partial p}(x,Du_-(x))\cdot D\varphi(x)d\sigma,
\]
which means that $\sigma$ is a solution of the continuity equation 
\[
\mathrm{div}\big(\sigma \frac{\partial K}{\partial p}(x,Du_-)\big)=0
\]
in the sense of distributions. 
In view of the above arguments, one can deduce  that if there is a Borel probability measure $m$ on $M$ such that $l(x,v)+F(x,m)$ admits a Mather measure $\eta_m$ with
\begin{align}\label{lab5}
m=\pi\sharp\eta_m,
\end{align}
then for any viscosity solution $u$ of 
\[
K(x,Du)=F(x,m)+c(m),
\]
where $c(m)$ is the Ma\~n\'e critical value of $K(x,p)-F(x,m)$, the pair $(u,m)$ is a solution of \eqref{lab1}, i.e.,  the Hamilton-Jacobi equation is satisfied in
viscosity sense and the continuity equation is satisfied in the sense of distributions. So, in order to find such a solution of \eqref{lab1}, it suffices to find  a probability measure $m$ satisfying \eqref{lab5}.

In this paper we aim to prove the existence of solutions of the following contact mean field games system
\begin{subequations}\label{1}
	\begin{empheq}[left=\empheqlbrace]{align}
		& H(x,u,Du)=F(x,m) \quad ~~~~~~~\text{in} \quad M,\\
		& \text{div}\Big(m \frac{\partial H}{\partial p}(x,u,Du)\Big)=0   \quad ~~\text{in} \quad M,\\
		& \int_{M}{m \ dx}=1
	\end{empheq}
\end{subequations}
using dynamical approaches. Note that the Hamiltonian $H=H(x,u,p)$ in \eqref{1} is defined on $T^*M\times \R$, where $(x,p)\in T^*M$ and $u\in\R$. For any given probability measure $m$ on $M$, equation (1.5a) can be interpreted as a hypersurface in the space $J^1(M,\R)$ of 1-jets of functions. Since the characteristic equations of (1.5a) is a contact Hamiltonian system, we call \eqref{1} a contact mean field games system. In view of the essential differences between weak KAM results for Hamiltonian and contact Hamiltonian systems, we cannot use the aforementioned idea directly to get the existence of solutions. A more careful analysis of the structure of Mather sets of contact Hamiltonian systems is needed.

We now list some basic assumptions on $H$ and $F$ which will be made in most of the results of this paper.

 Assume that the contact Hamiltonian $H$ is of class $C^3$ and satisfies:
\begin{itemize}
	\item [\textbf{(H1)}] \textbf{Positive Definiteness}: {\it For every $(x,u,p)\in T^*M\times\R$, the second partial derivative $\partial^2 H/\partial p^2 (x,u,p)$ is positive definite as a quadratic form;}
	\item [\textbf{(H2)}] \textbf{Superlinearity}: {\it For every $(x,u)\in M\times\R$, $H(x,u,p)$ is  superlinear in $p$;}
	\item [\textbf{(H3)}] \textbf{Strict Monotonicity}: {\it There are constants $\delta>0$ and $\lambda>0$ such that for every $(x,u,p)\in T^{\ast}M\times\R$,
		\begin{equation*}
			\delta<\frac{\partial H}{\partial u}(x,u,p)\leq\lambda.
		\end{equation*}
	}
\item [\textbf{(H4)}] \textbf{Reversibility}:{\it $H(x,u,p)=H(x,u,-p)$ for all $(x,u,p)\in T^*M\times\mathbb{R}$.}
\end{itemize}

We denote by $\mathcal{P}(M)$ the set of Borel probability measures on $M$, and by $\mathcal{P}(T^*M)$ the set of Borel probability measures on $T^*M$.
Both sets are endowed with
the weak-* convergence. A sequence $\{\mu_k\}_{k\in\mathbb{N}}\in \mathcal{P}(X)$ is weakly-* convergent to $\mu\in \mathcal{P}(X)$, denoted by $\mu_k \stackrel{w^*}{\longrightarrow} \mu$, if 
\[
\lim_{k\to\infty}\int_Xf(x)d\mu_k=\int_Xf(x)d\mu,\quad f\in C_b(X),
\]
where $C_b(X)$ denotes the function space of bounded uniformly continuous functions on $X$ with $X=M,\ T^*M$.
Let us recall that $\mathcal{P}(M)$ is compact for this topology. We shall work with the Monge-Wasserstein distance defined, for any $m_1$, $m_2\in\mathcal{P}(M)$, by
\[
d_1(m_1,m_2)=\sup_{h}\int_M h\ d(m_1-m_2),
\]
where the supremum is taken over all the maps $h:M\to\R$ which are 1-Lipschitz continuous. $\mathcal{P}_1(T^*M)$ denotes the Wassertein space of order 1, the space of probability measures with finite moment of order 1.

Let $F:M\times \mathcal{P}(M) \to \mathbb{R}$ be a function, satisfying the following  assumptions:
\begin{itemize}
	\item[\textbf{(F1)}] for every measure $m \in \mathcal{P}(M)$ the function $x \mapsto F(x,m)$ is of class $C^{2}(M)$ and	
	\begin{equation*}
		F_\infty:=\sup_{m \in \mathcal{P}(M)} \sum_{|\alpha|\leq 1} \| D^{\alpha}F(\cdot, m)\|_{\infty} < +\infty,
	\end{equation*} 
	where $\alpha=(\alpha_1,\cdots,\alpha_n)$,  $D^{\alpha}=D^{\alpha_1}_{x_1}\cdots D^{\alpha_n}_{x_n}$ and $\|\cdot\|_\infty$ denotes the supremum norm;
	\item[\textbf{(F2)}] for every $x \in M$ the function $m \mapsto F(x,m)$ is Lipschitz continuous and 
	$$\displaystyle{\sup_{\substack{x\in M\\ m_1,\ m_2 \in \mathcal{P}(M) \\ m_1\neq m_2 } }}\frac{|F(x,m_1)-F(x,m_2)|}{d_{1}(m_1, m_2)} < +\infty.$$
\end{itemize}

\begin{de}\label{ers}
A solution of the contact mean field games system \eqref{1} is a couple $(u, m) \in C(M) \times \mathcal{P}(M)$ such that (1.5a) is satisfied in distributions sense and (1.5b) is satisfied in viscosity sense.
\end{de}

The main result is stated as follows.
\begin{thm}\label{ma}
	Assume (H1)-(H4) and (F1), (F2). There exists at least one  solution $(u,m)$ of the contact mean field games system \eqref{1}, which has clear dynamical meaning. More precisely, there is a Mather measure $\mu_m$ for the contact Hamiltonian system \eqref{cm} such that $m=\pi_x\sharp\mu_m$, where $\pi_x:T^*M\times \R\to M$ denotes the canonical projection. 
\end{thm}

\begin{re} 
	\item[($i$)] Our methods depend on the analysis of dynamical behavior of the following contact Hamiltonian system:
	\begin{align}\label{cm}
		\left\{
		\begin{array}{l}
			\dot{x}=\frac{\partial H_m}{\partial p}(x,u,p),\\[3mm]
			\dot{p}=-\frac{\partial H_m}{\partial x}(x,u,p)-\frac{\partial H_m}{\partial u}(x,u,p)p,\\[3mm]
			\dot{u}=\frac{\partial H_m}{\partial p}(x,u,p)\cdot p-H_m(x,u,p),
		\end{array}
		\right.
	\end{align}
where $H_m(x,u,p):=H(x,u,p)-F(x,m)$ for all $(x,u,p)\in T^*M\times \R$.
\item [($ii$)] The notion of Mather measures was introduced by Mather in \cite{Mat91} for convex Hamiltonian systems, while the one for convex contact Hamiltonian systems was introduced in \cite{WWY2}, where part of Aubry-Mather and weak KAM theories for Hamiltonian systems was extended  to contact Hamiltonian systems under assumptions (H1), (H2) and strict increasing condition in the argument $u$. See \cite{MS,MS1} for Aubry-Mather and weak KAM results for discounted Hamiltonian systems.
\item[($iii$)] $u$ is Lipschitz and thus differential almost everywhere. Furthermore, $u$ is of class $C^{1,1}$ on the Mather set of system \eqref{cm}, and thus $Du(x)$ exists for $m$-a.e. $x \in M$ \cite[Proposition 4.2]{WWY2}.

\item[($iv$)] $m$ satisfies (1.5b) in the sense of distributions, that is,
$$ \int_{M}{\big\langle D\varphi(x), \frac{\partial H}{\partial p}\left(x,u(x), Du(x)\right) \big\rangle_x\ d m(x)}=0, \quad  \forall \varphi \in C^{1}(M). 
$$ 
\item [($v$)] Let $K(x,p)$ be a Tonelli Hamiltonian. Mean field games systems where the Hamiltonian has the following form
\[
\bar H(x,u,p)=u+K(x,p)
\]
appear in certain free-market economy models, see for instance \cite{Gomes}. It is clear that $\bar H(x,u,p)$ is a specific example of the Hamiltonians satisfying (H1)-(H3). To the best of our knowledge, Theorem \ref{ma} is the first step towards understanding contact mean field games systems from a dynamical point of view. 
\end{re}

See, for example, \cite{Car,Car1,Car2,Go} for recent progresses on first-order mean field games.


\section{Weak KAM results for Hamiltonian and contact Hamiltonian systems}

We recall some weak KAM type results for Tonelli Hamiltonian systems and Tonelli contact Hamiltonian systems. Results in Section 2.1 come from \cite{Fat-b}, and the ones in Section 2.2 come from \cite{WWY1,WWY2}.
\subsection{Weak KAM results for Hamiltonian systems}

\noindent $\bullet$ {\bf Ma\~n\'e's critical value.}\ 

Let $K$ denote a Tonelli Hamiltonian on $T^*M$ and let $l$ denote the associated Tonelli Lagrangian on $TM$ as in the Introduction.
If $[a, b]$ is a finite interval and $\gamma:[a, b] \rightarrow M$ is an absolutely continuous curve, we define its $l$ action as
$$
A_{l}(\gamma)=\int_{a}^{b} L(\gamma(s), \dot{\gamma}(s)) ds.
$$
The following estimate for action $A_l(\cdot)$ will be used later.
\begin{prop}(\cite[Proposition 4.4.4]{Fat-b})\label{bound}
For every given $t>0$, there exists a constant $C_{t}<+\infty$, such that, for each $x, y \in M$, there is a $C^{\infty}$ curve $\gamma:[0, t] \rightarrow M$ with $\gamma(0)=x, \gamma(t)=y$ and
$A_{l}(\gamma) \leq C_{t}$.
\end{prop}

The Ma\~n\'e critical value of the Lagrangian $l$, which was introduced by Ma\~n\'e in \cite{Man}, is defined by
$c(l):=\sup\{k \in \R: A_{l+k}(\gamma)<0$ for some absolutely continuous closed curve $\gamma\}$. 
The Ma\~n\'e's critical value has several other   respresentation formulas:
$$
c(l)=\inf _{u\in C^1(M)}\max_{x\in M}K(x,Du(x))=-\inf _{\mu} \int_{TM} l(x,v) d \mu,
$$
where the second infimum is taken with respect to all Borel probability measures on $TM$ invariant by the Euler-Lagrange flow $\Phi^l_t$. Furhtermore, $c(l)$ is the unique value of $e$ for which $K(x,Du)=e$ admits a  viscosity solution. In the following, we also call $c(l)$ the Ma\~n\'e critical value of the Hamiltonian $K$, denoted by $c(K)$.

\noindent $\bullet$ {\bf Weak KAM solutions.}\ 
A backward weak KAM solution  of equation \eqref{lab3} is a function $u: M \to \R$ such that
\begin{itemize}
	\item [(1)] $
	u(x)-u(y) \leq \inf_{s>0}\big\{\inf_{\alpha}A_l(\alpha)+c(K)s\big\},  \forall x, y \in M,
	$
	where  the second infimum is taken over all the absolutely continuous curves $\alpha$ :
	$[0, s] \to M$ with $\alpha(0)=y$ and $\alpha(s)=x;$ 
	\item [(2)] for every $x \in M$ there exists a curve $\gamma_{x}:(-\infty, 0] \to M$ with $\gamma_{x}(0)=x$ such that
	$$
	u(x)-u(\gamma(t))=\int_{t}^{0} l\left(\gamma_{x}(s), \dot{\gamma}_{x}(s)\right) d s-c(K)t, \quad \forall t \in(-\infty, 0].
	$$
\end{itemize}
Fathi introduced this notion and showed that backward weak KAM solutions and viscosity solutions of equation \eqref{lab3} are the same \cite[Theorem 7.6.2]{Fat-b}. Similarly, one can define forward weak KAM solutions of equation \eqref{lab3}.

\begin{prop}(\cite[Proposition 4.2.1]{Fat-b})\label{li}
The family of viscosity solutions of equation \eqref{lab3} is equi-Lipschitz with the Lipschitz constant  $\mathrm{Lip}(u)\leq B+c(K)$, where
\[
B =\sup \left\{l(x, v): (x, v) \in TM,\|v\|_{x}=1\right\},
\]
where $\|\cdot\|_x$ denotes the norm on $T_xM$ induced by a Riemannian metric.
\end{prop}

\subsection{Weak KAM results for contact Hamiltonian systems}
\noindent $\bullet$ {\bf Admissibility.} \  
Let $H(x,u,p)$ be a contact Hamiltonian satisfying (H1), (H2) and $|\frac{\partial H}{\partial u}|\leq \lambda$ for some $\lambda>0$. Then there exists at least a real number $c$ such that $H(x,u,Du)=c$ admits viscosity solutions \cite{WWY1}.  Furthermore, there may be two different real numbers $c_1$ and $c_2$ such that $H(x,u,Du)=e$ admits viscosity solutions with $e=c_1$, $c_2$, respectively \cite{WWY1}. So, from this point of view, the notion of Ma\~n\'e's critical value is inapplicable to contact Hamiltonian systems or contact Hamilton-Jacobi equations. A new concept is needed to guarantee the existence of viscosity solutions of $H(x,u,Du)=e$. We say that $H(x,u,p)$ is admissible \cite{WWY2}, if there exists $a\in \R$ such that $c(H(x,a,p))=0$, where $c(H(x,a,p))$ denotes the Ma\~n\'e critical value of the classical Hamiltonian $H(x,a,p)$.

For $H(x,u,p)$ satisfying (H1), (H2) and $0\leq\frac{\partial H}{\partial u}\leq \lambda$, 
it was proven in \cite[Appendix B]{WWY2} that 
\begin{align}\label{4-1}
	H(x,u,Du)=0
\end{align} 
has viscosity solutions if and only if $H(x,u,p)$ is admissible. When $H(x,u,p)$ satisfies stronger assumptions (H1), (H2) and (H3), one can deduce that $H(x,u,p)$ is admissible \cite[Remark 1.2]{WWY2}, and thus equation \eqref{4-1} has viscosity solutions. In fact, under (H1)-(H3) equation \eqref{4-1} has a unique viscosity solution \cite[Appendix A]{WWY2}, denoted by $u_-$. Moreover, $u_-$ is Lipschitz on $M$. 

From now on to the end of Section 2, we always assume (H1)-(H3).

\noindent $\bullet$ {\bf Backward weak KAM solutions and calibrated curves.}\ 
Let $\Phi^H_{t}$ denote the local flow of \begin{align}\label{c}
	\left\{
	\begin{array}{l}
		\dot{x}=\frac{\partial H}{\partial p}(x,u,p),\\[3mm]
		\dot{p}=-\frac{\partial H}{\partial x}(x,u,p)-\frac{\partial H}{\partial u}(x,u,p)p,\\[3mm]
		\dot{u}=\frac{\partial H}{\partial p}(x,u,p)\cdot p-H(x,u,p).
	\end{array}
	\right.
\end{align}
The Legendre transform $\mathcal{L}: T^{*} M \times \mathbb{R} \rightarrow T M \times \mathbb{R}$ defined by
$$
\mathcal{L}: (x, u, p) \mapsto\left(x, u, \frac{\partial H}{\partial p}(x, u, p)\right)
$$
is a diffeomorphism. Using $\mathcal{L}$, we can define the contact Lagrangian $L(x, u,v)$ associated to $H(x, u, p)$ as
$$
L(x, u, v):=\sup _{p \in T_{x}^{*} M}\left\{\langle v, p\rangle_{x}-H(x, u, p)\right\}.
$$
Then $L(x, u, v)$ and $H(x, u, p)$ are Legendre transforms of each other, depending on conjugate variables $v$ and $p$ respectively. Let $\Phi^L_t=\mathcal{L}\circ \Phi^H_t\circ \mathcal{L}^{-1}$. We call $\Phi^L_t$ the Euler-Lagrange flow.

Following Fathi, one can define backward weak KAM solutions of equation \eqref{4-1} as follows. 
A function $u \in C(M)$ is called a backward weak KAM solution if:  
(i) for each continuous piecewise $C^{1}$ curve $\gamma:\left[t_{1}, t_{2}\right] \rightarrow M$, 
$$
u\left(\gamma\left(t_{2}\right)\right)-u\left(\gamma\left(t_{1}\right)\right) \leq \int_{t_{1}}^{t_{2}} L(\gamma(s), u(\gamma(s)), \dot{\gamma}(s)) ds;
$$
(ii) for each $x \in M$, there exists a $C^{1}$ curve $\gamma_x:(-\infty, 0] \rightarrow M$ with $\gamma_x(0)=x$ such that
$$
u(x)-u(\gamma_x(t))=\int_{t}^{0} L(\gamma_x(s), u(\gamma_x(s)), \dot{\gamma}_x(s)) ds, \quad \forall t<0.
$$
Backward weak KAM solutions and viscosity solutions are still the same \cite[Proposition 2.7]{WWY2}. Thus, equation \eqref{4-1} has a unique backward weak KAM solution $u_-$. The curves in (ii) are called $(u, L, 0)$-calibrated curves. We can also define forward weak KAM solutions of equation \eqref{4-1}. Note that backward and forward weak KAM solutions of equation \eqref{lab3} always exist in pairs \cite[Theorem 5.1.2]{Fat-b}. But, this is not the case for equation \eqref{4-1}, see for instance, \cite[Example 1.1]{WWY2}.

\begin{prop}(\cite[Proposition 4.1]{WWY2})\label{solu}
	Given $x \in M ,$ if $\gamma:(-\infty, 0] \rightarrow M$ is a $(u_-, L, 0)$-calibrated curve with $\gamma(0)=x$, then $(\gamma(t), u_-(\gamma(t)), p(t))$ satisfies equations \eqref{c} on $(-\infty, 0)$, where $p(t)=\frac{\partial L}{\partial v}(\gamma(t), u_-(\gamma(t)), \dot{\gamma}(t)) .$ 
\end{prop}

Let us recall two semigroups of operators introduced in \cite{WWY1}. Define a family of nonlinear operators $\left\{T^-_{t}\right\}_{t \geq 0}$ from $C(M)$ to itself as follows. For each $\varphi \in C(M)$, denote by $(x, t) \mapsto T^-_{t} \varphi(x)$ the unique continuous function on $(x,t) \in M \times[0,+\infty)$ such that
$$
T^-_{t} \varphi(x)=\inf _{\gamma}\left\{\varphi(\gamma(0))+\int_{0}^{t} L\left(\gamma(\tau), T^-_{\tau} \varphi(\gamma(\tau)), \dot{\gamma}(\tau)\right) d \tau\right\},
$$
where the infimum is taken among absolutely continuous curves $\gamma:[0, t] \rightarrow M$ with $\gamma(t)=x .$ $\{T^-_t\}_{t\geq 0}$ is called the backward solution semigroup. The infimum can be achieved. 
Similarly, one can define another semigroup of operators $\left\{T_{t}^{+}\right\}_{t\geq 0}$, called the forward solution semigroup, by
$$
T_{t}^{+} \varphi(x)=\sup _{\gamma}\left\{\varphi(\gamma(t))-\int_{0}^{t} L\left(\gamma(\tau), T_{t-\tau}^{+} \varphi(\gamma(\tau)), \dot{\gamma}(\tau)\right) d \tau\right\},
$$
where the infimum is taken among absolutely continuous curves $\gamma:[0, t] \rightarrow M$ with $\gamma(0)=x$. These two semigroups can be regarded as the contact counterparts of the Lax-Oleinik semigroups for classical Lagrangians defined on $TM$.

\begin{prop}(\cite[Proposition 2.9, 2.10]{WWY2})\label{conv}
	For each $\varphi \in C(M)$, the uniform limit $\lim _{t \rightarrow+\infty} T^-_{t} \varphi$ exists and $\lim _{t \rightarrow+\infty} T^-_{t} \varphi=u_-$. The uniform limit $\lim _{t \rightarrow+\infty} T_{t}^{+} u_{-}$ exists. Let $u_{+}=\lim _{t \rightarrow+\infty} T_{t}^{+} u_{-} .$ Then $u_{+}$ is a forward weak KAM solution of equation \eqref{4-1}.
\end{prop}
See \cite{Fat-b} for the convergence result for the Lax-Oleinik semigroups associated with autonomous Lagrangian systems $l(x,v)$, and \cite{WY} for the convergence result for a kind of modified Lax-Oleinik semigroups associated with time-periodic  Lagrangian systems $g(t,x,v)$.

\begin{prop}(\cite[Lemma 4.7]{WWY2})\label{pr803}
	 Let $u_+$ be as in Proposition \ref{conv}. 
	 For any given $x \in M$ with $u_{-}(x)=u_{+}(x)$, there exists a curve $\gamma:(-\infty,+\infty) \rightarrow M$ with $\gamma(0)=x$ such that $u_{-}(\gamma(t))=u_{+}(\gamma(t))$ for each $t \in \mathbb{R}$, and
	$$
	u_{\pm}\left(\gamma\left(t^{\prime}\right)\right)-u_{\pm}(\gamma(t))=\int_{t}^{t^{\prime}} L\left(\gamma(s), u_{\pm}(\gamma(s)), \dot{\gamma}(s)\right) d s, \quad \forall t \leq t^{\prime} \in \mathbb{R}
	$$
	Moreover, $u_{\pm}$ are differentiable at $x$ with the same derivative $D u_{\pm}(x)=$ $\frac{\partial L}{\partial v}\left(x, u_{\pm}(x), \dot{\gamma}(0)\right)$.
\end{prop}

\begin{prop}(\cite[Theorem 1.5]{WWY2})\label{long}
	Given $x\in M$, let $\gamma:(-\infty, 0] \rightarrow M$ be a $\left(u_{-}, L, 0\right)$-calibrated curve with $\gamma(0)=x.$ Let $u:=$ $u_{-}\left(x\right), p:=\frac{\partial L}{\partial \dot{x}}(x, u, \dot{\gamma}(0)_{-})$ where $\dot{\gamma}(0)_{-}$ denotes the left derivative of $\gamma(t)$ at
	$t=0 .$ Let $\alpha\left(x, u, p\right)$ be the $\alpha$-limit set of $\left(x, u, p\right) .$ Then
	$$
	\alpha\left(x, u, p\right) \subseteq \tilde{\mathcal{A}}_H,
	$$
	where $\tilde{\mathcal{A}}_H$ denotes the Aubry set for system \eqref{c}.
\end{prop}
We recall the definitions of Mather sets and Aubry sets for \eqref{c} now.

\noindent $\bullet$ {\bf Aubry and Mather sets.}\ 
We define a subset of $T^{*} M \times \mathbb{R}$ associated with $u_{-}$ by
$G_{u_{-}}:=\operatorname{cl}\left(\left\{(x, u, p): x\right.\right.$ is a point of differentiability of $\left.\left.u_{-}, u=u_{-}(x), p=D u_{-}(x)\right\}\right)$, 
where $\operatorname{cl}(S)$ denotes the closure of $S \subseteq T^{*} M \times \mathbb{R} .$

Define the Aubry set for \eqref{c} by
$$
\tilde{\mathcal{A}}_H:=\bigcap_{t \geq 0} \Phi^H_{-t}\left(G_{u_{-}}\right).
$$
$\tilde{\mathcal{A}}_H$ is non-empty, compact and $\Phi^H_t$-invariant \cite{WWY2}. Then there exist Borel $\Phi^H_{t}$ -invariant probability measures supported in $\tilde{\mathcal{A}}_H .$ We call these measures Mather measures and denote by $\mathfrak{M}$ the set of Mather measures. The Mather set is defined by
$$
\tilde{\mathcal{M}}_H=\mathrm{cl}\left(\bigcup_{\mu \in \mathfrak{M}} \operatorname{supp}(\mu)\right).
$$

We call $\mathcal{A}_H:=\pi_x \tilde{\mathcal{A}}_H$ and $\mathcal{M}_H:=\pi_x \tilde{\mathcal{M}}_H$, the projected Aubry set and the projected Mather set, respectively. The projection $\pi_x: T^{*} M \times \mathbb{R} \to M$ induces a bi-Lipschitz homeomorphism from $\tilde{\mathcal{A}}_H$ to $\mathcal{A}_H$ \cite[Theorem 1.3]{WWY2}. 
We also have that \cite[Theorem 1.3, Formula (1.15)]{WWY2}
\begin{align}\label{bh}
\tilde{\mathcal{A}}_H=G_{u_-}\cap G_{u_+}\subset\left\{(x, u, p) \in T^{*} M \times \mathbb{R}: H(x, u, p)=0, u=u_{-}(x)\right\}, 
\end{align}
where $u_+$ is as in Proposition \ref{conv}, and $G_{u_{+}}:=\operatorname{cl}\left(\left\{(x, u, p): x\right.\right.$ is a point of differentiability of $\left.\left.u_{+}, u=u_{+}(x), p=D u_{+}(x)\right\}\right).$

We will also use the following notations
\[
\tilde{\mathcal{M}}_L:=\mathcal{L}(\tilde{\mathcal{M}}_H)\subset TM,\quad \tilde{\mathcal{A}}_L:=\mathcal{L}(\tilde{\mathcal{A}}_H)\subset TM.
\]

\section{Existence of slutions of mean field games system}

\subsection{Mather sets of reversible contact Hamiltonian systems}

Under the assumptions (H1)-(H4), we can take a closer look at the Mather set of \eqref{c}. 
\begin{prop}\label{key2}
Let 
\[
\mathcal{K}_H:=\{(x,u_-(x),0):H(x,u_-(x),0)=0\}.
\]
Then $\mathcal{K}_H$ is a non-empty compact subset of the Mather set $\tilde{\mathcal{M}}_{H}$.
\end{prop}

\begin{proof}
Since $H$ is strictly convex and reversible in $p$, then by the Legendre transform we have  
\[
L(x,u_-(x),0)=\sup_{p\in T^*_xM}-H(x,u_-(x),p)=-\inf _{p\in T^*_xM}H(x,u_-(x),p)=-H(x,u_-(x),0),\ \forall x\in M.
\]
For any $(x,u_-(x),0)\in \mathcal{K}_H$, let $\g^*(t)\equiv x$ for $t\leq 0$. Then $u_-(\g^*(t))\equiv u_-(x)$ and $\dot{\g}^*(t)\equiv 0$. 
Note that 
\[
\int_{t}^0L(\g^*(s),u_-(\g^*(s)),\dot{\g}^*(s))\ ds=\int_{t}^0L(x,u_-(x),0)\ ds=0,\quad \forall t\leq 0.
\]
Thus, we get that 
\[
u_-(x)-u_-(\g^*(t))=\int_{t}^0L(\g^*(s),u_-(\g^*(s)),\dot{\g}^*(s))\ ds,\quad \forall t\leq 0,
\]
implying that $\g^*$ is a $(u_-,L,0)$-calibrated curve with $\g^*(0)=x$. Thus, by Proposition \ref{solu} we get that
\[
\Big(\g^*(t),u_-(\g^*(t)),\frac{\partial L}{\partial v}\big(\g^*(t),u_-(\g^*(t)),\dot{\gamma}^*(t)\big)\Big)=(x,u_-(x),0),\quad \forall t\leq 0
\]
satisfies equation \eqref{c}. 
It means that $(x,u_-(x),0)$ is a fixed point of the flow $\Phi^H_t$. From Proposition \ref{long}, the $\alpha$-limit set of $(x,u_-(x),0)$ is contained in the Aubry set $\tilde{\mathcal{A}}_{H}$, and thus we deduce that $(x,u_-(x),0)\in \tilde{\mathcal{A}}_{H}$. Since each point in $\mathcal{K}_H$ is a fixed point, then $\mathcal{K}_H\subset\tilde{\mathcal{M}}_{H}$.

Next we show that $\mathcal{K}_H$ is non-empty. Assume by contradiction that $\mathcal{K}_H=\emptyset$.
	There would be two possibilities: 
	(Case I) $H(x,u_-(x),0)>0$,  $\forall x\in M$; 
	(Case II) $H(x,u_-(x),0)<0$, $\forall x\in M$.

({\bf Case I}): for any $(x,p)\in T^*M$, since $H(x,u_-(x),p)\geq H(x,u_-(x),0)>0$, then the set
\[
\{(x,u_-(x),p):  H(x,u_-(x),p)=0\}=\emptyset.
\] 
Recall \eqref{bh}, i.e.,  $\tilde{\mathcal{A}}_H\subset \{(x,u_-(x),p):  H(x,u_-(x),p)=0\}$. Since  $\tilde{\mathcal{A}}_H$ is non-empty, then $\{(x,u_-(x),p):  H(x,u_-(x),p)=0\}$ is non-empty, a contradiction.

({\bf Case II}): Since 
\[
L(x,u_-(x),0)=-H(x,u_-(x),0)>0,\quad \forall x\in M,
\]
then 
\begin{align}\label{contrad}
	L(x,u_-(x),v)\geq L(x,u_-(x),0)>0,\quad \forall (x,v)\in TM.
\end{align}
Recall that the Mather set $\tilde{\mathcal{M}}_L\subset \tilde{\mathcal{A}}_L$ is non-empty. Taking an arbitrary $(x_0,u_0,v_0)\in \tilde{\mathcal{M}}_L$, let $(x(t),u(t),\dot{x}(t))=\Phi^L_t(x_0,u_0,v_0)$ for $t\in\mathbb{R}$. Then $x(t)$ is a $(u_-,L,0)$-calibrated curve implying that 
\begin{align}\label{cali}
u_-(x(t))-u_-(x_0)=\int_0^tL(x(s),u_-(x(s)),\dot{x}(s))\ ds,\quad \forall t>0.	
\end{align}
In fact, since $(x_0,u_0,v_0)\in \tilde{\mathcal{M}}_L$, then by the definition of the Aubry set and   \eqref{bh}, one can deduce that  $u_0=u_-(x_0)$ and $v_0=\frac{\partial H}{\partial p}(x_0,u_-(x_0),Du_-(x_0))$.
In view of Proposition \ref{pr803}, there is a $(u_-,L,0)$-calibrated curve $\gamma:(-\infty,+\infty)\to M$ with $\gamma(0)=x_0$ and $\dot{\gamma}(0)=\frac{\partial H}{\partial p}(x_0,u_-(x_0),Du_-(x_0))$. Then by Proposition \ref{solu}, $$(\gamma(t),u_-(\gamma(t)),\frac{\partial H}{\partial p}(\gamma(t),u_-(\gamma(t)),Du_-(\gamma(t)))=\Phi^L_t(x_0,u_0,v_0)$$ for all $t\in\R$. Thus, $x(t)=\gamma(t)$ for all $t\in\R$.

Since  $(x_0,u_0,v_0)\in \tilde{\mathcal{M}}_L$, then $(x_0,u_0,v_0)$ belongs to the support of some $\Phi^L_t$-invariant probability measure and thus by Poincar\'e's recurrence theorem we get $(x_0,u_0,v_0)\in \omega\big(x_0,u_0,v_0\big)$, where $\omega\big(x_0,u_0,v_0\big)$ denotes the $\omega$-limit set of the orbit $(x(t),u(t),\dot{x}(t))$. Thus, there exist $\{t_n\}_{n\in \mathbb{N}}$ with $t_n\to +\infty$, such that $|x(t_n)-x_0|\leq \frac{1}{n}$. By \eqref{cali}, we deduce that
\begin{align}\label{contra}
\int_0^{t_n}L(x(s),u_-(x(s)),\dot{x}(s))\ ds=|u_-(x(t_n))-u_-(x_0)|\leq K_{u_-}|x(t_n)-x_0|\leq \frac{K_{u_-}}{n},
\end{align}
where $K_{u_-}>0$ is the Lipschitz constant of $u_-$. For $n$ large enough, combining \eqref{contrad} and \eqref{contra} leads to a contradiction.

Hence, $\mathcal{K}_H$ is non-empty. In view of the compactness of $M$, it is clear that $\mathcal{K}_H$ is also compact. 
\end{proof}

\begin{prop}\label{Mat}
$\mathcal{K}_H=\tilde{\mathcal{M}}_H$.	
\end{prop}

\begin{proof}
	Since $H(x,u,p)$ is strictly convex in $p$ and $H(x,u,p)=H(x,u,-p)$ for all $(x,u,p)\in T^*M\times\R$, then it is direct to see that $\frac{\partial H}{\partial p}(x,u,p)\cdot p\geq 0$, where equality holds if and only if $p=0$. 
	
	Let $(x(t),u(t),p(t))$ be an arbitrary trajectory in the Mather set $\tilde{\mathcal{M}}_H$. Then in view of $\tilde{\mathcal{M}}_H\subset \tilde{\mathcal{A}}_H\subset G_{u_-}$ and the differentiability of $u_-$ on $\mathcal{A}_H$, we deduce that $$(x(t),u(t),p(t))=(x(t),u_-(x(t)),Du_-(x(t))).$$
	Note that
	\begin{align*}
	\frac{d}{dt}u_-(x(t))&=\frac{\partial H}{\partial p}(x(t),u_-(x(t)),Du_-(x(t)))\cdot Du_-(x(t))-H(x(t),u_-(x(t)),Du_-(x(t)))\\
	&=\frac{\partial H}{\partial p}(x(t),u_-(x(t)),Du_-(x(t)))\cdot Du_-(x(t)).
	\end{align*}
Then $\frac{d}{dt}u_-(x(t))\geq 0$ and $\frac{d}{dt}u_-(x(t))=0$ if and only if $Du_-(x(t))=0$. If
there is $t_0\in\R$ such that $Du_-(x(t_0))=0$, then in view of the proof of the above proposition, $(x(t_0),u_-(x(t_0)),Du_-(x(t_0)))$ is a fixed point of $\Phi^H_t$. If $Du_-(x(t))\neq0$ for all $t\in\R$, then
$\frac{d}{dt}u_-(x(t))>0$ for all $t\in\R$, which contradicts the recurrence property of points in the Mather set $\tilde{\mathcal{M}}_H$. Hence, one can deduce that $\tilde{\mathcal{M}}_H$ consists of fixed points which have the form $(x,u_-(x),Du_-(x))$ with $H(x,u_-(x),Du_-(x))=0$. So far, we have proved that $\tilde{\mathcal{M}}_H\subset \mathcal{K}_H$, which together with Proposition \ref{key2} finishes the proof.
\end{proof}

\subsection{Proof of Theorem \ref{ma}}

For each $m\in\mathcal{P}(M)$, $H_m(x,u,p):=H(x,u,p)-F(x,m)$ is a Hamiltonian on $T^*M\times \R$. When $H$ satisfies (H1)-(H4), so does $H_m$. Thus, $H_m$ is admissible for all $m\in \mathcal{P}(M)$. Let $a_m\in\mathbb{R}$ be such that the Ma\~n\'e critical value of $H(x,a_m,p)-F(x,m)$ is $0$, that is,
\[
0=\inf_{x\in M}\Big(L(x,a_m,0)+F(x,m)\Big)=-\sup_{x\in M}\Big(H(x,a_m,0)-F(x,m)\Big). 
\]

\begin{lem}\label{b1}
There is a constant $D_1>0$ such that $|a_m|\leq D_1$ for all $m\in\mathcal{P}(M)$.
\end{lem}

\begin{proof}
For any $x\in M$, any $a\in\mathbb{R}$, 
\[
H(x,a,0)=H(x,0,0)+\frac{\partial H}{\partial u}(x,\theta,0)a,
\]
for some $\theta\in\mathbb{R}$ depending on $x$ and $a$. Recall that $0<\delta\leq \frac{\partial H}{\partial u}(x,u,p)$ for all $(x,u,p)\in T^*M\times\mathbb{R}$.

If $a>0$, then $H(x,a,0)\geq H(x,0,0)+\delta a$. So, we deduce that $\lim_{a\to+\infty}H(x,a,0)=+\infty$ uniformly in $x\in M$. If $a<0$, then $H(x,a,0)\leq H(x,0,0)+\delta a$. So, we deduce that $\lim_{a\to-\infty}H(x,a,0)=-\infty$ uniformly in $x\in M$.

By the definition of $a_m$, we get that
\[
0=\sup_{x\in M}(H(x,a_m,0)-F(x,m)).
\]
By the above arguments and the boundedness of $F$, it is clear that the set $\{a_m\}_{m\in\mathcal{P}(M)}$ is bounded.
\end{proof}

\begin{re}
Lemma \ref{b1} still holds ture when $H$ satisfies (H1)-(H3).
In fact, since $a_m$ satisfies 
\[
0=\inf_{u\in C^1(M)}\max_{x\in M}(H(x,a_m,Du(x))-F(x,m)),
\]	
then 
\begin{align}\label{404}
0\leq \max_{x\in M}(H(x,a_m,0)-F(x,m)).
\end{align}
On the other hand, for any $u\in C^1(M)$, there must be a point $x_u\in M$ such that $Du(x_u)=0$ since $M$ is compact and closed.
Thus, we have that
\[
\max_{x\in M}(H(x,a_m,Du(x))-F(x,m))\geq H(x_u,a_m,0)-F(x_u,m),
\]
implying
\begin{align}\label{405}
0\geq \min_{x\in M}(H(x,a_m,0)-F(x,m)).
\end{align}
By similar arguments used in the proof of Lemma \ref{b1}, we can deduce from \eqref{404}, \eqref{405} and (H3) that $\{a_m\}_{m\in\mathcal{P}(M)}$ is bounded.
\end{re}

The following result is a direct consequence of Proposition \ref{li} and Lemma \ref{b1}.
\begin{lem}\label{b2}
For each $m\in\mathcal{P}(M)$, let $w_m$ denote an arbitrary  viscosity solution of 
\begin{align}\label{hjm}
H(x,a_m,Dw)-F(x,m)=0.
\end{align} 
Then $\{w_m\}_{m\in\mathcal{P}(M)}$ is equi-Lipschitz with a Lipschitz constant $D_2>0$ given by 
\[
D_2:=\sup\{L(x,u,v)+F_\infty: (x,u,v)\in TM\times\mathbb{R},\ |u|\leq D_1, \ \|v\|_x=1\}.
\]
\end{lem}

Define 
\[
h^m_t(x,y):=\inf_{\g}\int_0^t\Big(L(\g(s),a_m,\dot{\g}(s))+F(\g(s),m)\Big)\ ds,\quad \forall x,y\in M,
\]
where the infimum is taken among the absolutely continuous curves $\g:[0,t]\to M$ with $\g(0)=x$ and $\g(t)=y$. 
By definition and Lemma \ref{b2},  for any $x$, $y\in M$ and any $t>0$, we deduce that
\begin{align}\label{b4}
h^m_t(x,y)\geq w_m(y)-w_m(x)\geq -D_2\mathrm{diam}(M), \quad \forall m\in\mathcal{P}(M),
\end{align}
which means that $h^m_t(x,y)$ is bounded from below.

The proof of the following lemma  is quite similar to the one of Proposition \ref{bound}, thus we omit it here.
\begin{lem}\label{b3}
For each given $t>0$, there is a constant $E_t\in\mathbb{R}$ such that for any $x$, $y\in M$, there is  a $C^\infty$ curve $\g:[0,t]\to M$ with $\g(0)=x$, $\g(t)=y$ and 
\[
\int_0^t\Big(L(\g(s),a_m,\dot{\g}(s))+F(\g(s),m)\Big)\ ds\leq E_t,\quad \forall m\in\mathcal{P}(M),
\]
where $E_t$ is given by
\[
 E_t:=t\tilde{E}_t, \quad \tilde{E}_t:=\sup \Big\{L(x,u,v)+F_\infty: (x,u,v)\in TM\times\mathbb{R},\ |u|\leq D_1, \ \|v\|_x\leq \frac{\mathrm{diam}(M)}{t}\Big\}.
 \]
\end{lem}

Let $w_m'(x):=w_m(x)-w_m(0)$， where $w_m$ is as in Lemma \ref{b2}. Then $w_m'$ is still a viscosity solution of \eqref{hjm} and $w_m'(0)=0$.
From Lemma \ref{b2},  Lemma \ref{b3} and \cite[Lemma 5.3.2 (4)]{Fat-b}, one can deduce that for any given $t_0>0$, if $t\geq t_0$, then 
\begin{align}\label{b5}
h^m_t(x,y)\leq E_{t_0}+2\|w_m'\|_\infty\leq E_{t_0}+2D_2\mathrm{diam}(M),\quad \forall x,\ y\in M, \quad \forall m\in\mathcal{P}(M).
\end{align}

Based on \eqref{b4} and \eqref{b5}, we can get the following result
\begin{prop}\label{b6}
Given any $t_0>0$, for any $\phi\in C(M)$, there is a constant $D_{t_0,\phi}>0$ such that 
\[
|T^m_t\phi(x)|\leq D_{t_0,\phi},\quad \forall (x,t)\in M\times[t_0,+\infty),\quad \forall m\in\mathcal{P}(M),
\] 	
 where $\{T^m_t\}_{t\geq 0}$ denotes the backward solution semigroup associated with $L(x,u,v)+F(x,m)$.
\end{prop}

\begin{proof} Let $T^{a_m}_t$ denote the backward Lax-Oleinik operator associated with $L(x,a_m,v)+F(x,m)$, i.e., for each $\varphi\in C(M)$ and each $t\geq 0$, 
	\[
	T^{a_m}_t\varphi(x):=
	\inf _{\gamma}\left\{\varphi(\gamma(0))+\int_{0}^{t} \big(L\left(\gamma(\tau),a_m, \dot{\gamma}(\tau)\right)+F(\gamma(\tau),m)\big) d \tau\right\},
	\]
	where the infimum is taken among absolutely continuous curves $\gamma:[0, t] \to M$ with $\gamma(t)=x .$ The infimum can be achieved.

{\em Boundedness from above}:  
for $(x,t)\in M\times[t_0,+\infty)$ with $T^m_t\phi(x)>a_m$, let $\gamma:[0,t]\to M$ be a minimizer of $T^{a_m}_t\phi(x)$. Consider the function $s\mapsto T^{m}_s\phi(\g(s))$ for $s\in(0,t]$. 
Since $T^m_0\phi(\g(0))=\phi(\g(0))$ and $T^m_t\phi(x)>a_m$, then there exists $s_0\in[0,t)$ such that $T^m_{s_0}\phi(\g(s_0))\leq \max\{\phi(\g(0)),a_m\}$ and 
$T^m_{s}\phi(\g(s))>a_m$ for $s\in(s_0,t]$. 
Hence, by (H3), \eqref{b4} and \eqref{b5}, we have that
\begin{align*}
T^m_t\phi(x) &\leq T^m_{s_0}\phi(\g(s_0))+\int_{s_0}^t\big(L(\gamma(s),T^m_{s}\phi(\g(s)),\dot{\gamma}(s))+F(\gamma(s),m)\big)ds\\
&\leq \max\{\phi(\g(0)),a_m\}+\int_{s_0}^t\big(L(\gamma(s),a_m,\dot{\gamma}(s))+F(\gamma(s),m)\big)ds\\
&\leq\|\phi\|_\infty+D_1+h^m_{t-s_0}(\gamma(s_0),x)\\
&\leq \|\phi\|_\infty+D_1+E_{t_0}+3D_2\mathrm{diam}(M).
\end{align*}
We have proved that $T^m_t\phi(x)$ is bounded from above by $\|\phi\|_\infty+2D_1+E_{t_0}+3D_2\mathrm{diam}(M)$ on $M\times[t_0,+\infty)$.

{\em Boundedness from below}: For $(x,t)\in M\times[t_0,+\infty)$ with $T^m_t\phi(x)<a_m$, let $\alpha:[0,t]\to M$ be a minimizer of $T^{m}_t\phi(x)$. 
Consider the function $s\mapsto T^m_s\phi(\alpha(s))$ for $s\in(0,t]$. 
Since $T^m_0\phi(\alpha(0))=\phi(\alpha(0))$ and $T^m_t\phi(x)<a_m$, 
then there exists $s_0\in[0,t)$ such that $T^m_{s_0}\phi(\alpha(s_0))\geq \min\{\phi(\alpha(s_0)),a_m\}$ and $T^m_s\phi(\alpha(s))<a_m$ for $s\in(s_0,t]$. Hence, by (H3) and \eqref{b4}, we have that
\begin{align*}
T^m_t\phi(x) &=T^m_{s_0}\phi(\alpha(s_0))+\int_{s_0}^t\big(L(\alpha(s),T^m_s\phi(\alpha(s)),\dot{\alpha}(s))+F(\alpha(s),m)\big)ds\\
&\geq \min\{\phi(\alpha(s_0)),a_m\}+\int_{s_0}^t\big(L(\alpha(s),a_m,\dot{\alpha}(s))+F(\alpha(s),m)\big)ds\\
&\geq-\|\phi\|_\infty-D_1+h^m_{t-s_0}(\alpha(s_0),x)\\
&\geq -\|\phi\|_\infty-D_1-D_2\mathrm{diam}(M),
\end{align*}
which shows that $T^m_t\phi(x)$ is bounded from below by $-\|\phi\|_\infty-2D_1-D_2\mathrm{diam}(M)$ on $M\times[t_0,+\infty)$.
	
\end{proof}

By Proposition \ref{conv}, for each $m\in\mathcal{P}(M)$,  the uniform limit of $T^m_t\phi$ as $t\to\infty$ exists and the limit function is the unique viscosity solution $u_m$ of 
\[
H(x,u,Du)-F(x,m)=0.
\] 
Therefore, by Proposition \ref{b6}, there is a constant $D_3>0$ such that
\begin{align}\label{b7}
\|u_m\|_\infty\leq D_3,\quad \forall m\in\mathcal{P}(M).
\end{align}
Note that $F$ is bounded, then by the above estimate and \cite[Lemma 4.1]{WWY2}, we deduce that $\{u_m\}_{m\in\mathcal{P}(\T)}$ is equi-Lipschitz with a Lipschitz constant 
\[\sup\{L(x,u,v)+F_\infty: (x,u,v)\in TM\times\mathbb{R},\ |u|\leq D_3,\ \|v\|_x=1\}.
\]

\begin{prop}\label{key3}
For any $m\in\mathcal{P}(M)$, let $u_m$ denote the unique viscosity solution of 
\[
H(x,u,Du)=F(x,m),\quad x\in M.
\]	
Let $m_j$, $m_0\in \mathcal{P}(M)$, $j\in\mathbb{N}$. 
If $m_j\stackrel{w^*}{\longrightarrow} m_0$, as $j\to\infty$, then $u_{m_j}$ converges uniformly to $u_{m_0}$ on $M$, as $j\to\infty$.
\end{prop}

\begin{proof}
	Let $H_m(x,u,p):=H(x,u,p)-F(x,m)$. Then by (F2) and $m_j\stackrel{w^*}{\longrightarrow} m_0$ as $j\to\infty$, $H_{m_j}$ converges uniformly to $H_{m_0}$ on compact subsets of $T^*M\times\mathbb{R}$, as $j\to\infty$. Since $\{u_m\}_{m\in\mathcal{P}(M)}$ is uniformly bounded and equi-Lipschitz, then by the stability of viscosity solutions and the uniqueness of viscosity solutions of 
	\[
H(x,u,Du)=F(x,m_0),\quad x\in M,
\]
we conclude that $u_{m_j}$ converges uniformly to $u_{m_0}$ on $M$, as $j\to\infty$.
\end{proof}

\begin{re}
	Let us point out that Lemmas \ref{b1}, \ref{b2}, \ref{b3} and Propositions \ref{b6}, \ref{key3} still hold true under assumptions (H1)-(H3).
\end{re}

It is a position to give the proof of the main result of this paper. 

\begin{proof}[Proof of Theorem \ref{ma}]
For any $m\in \mathcal{P}(M)$, in view of Proposition \ref{Mat}, we know that
\[
\tilde{\mathcal{M}}_{H_m}=\mathcal{K}_{H_m}=\{(x,u_m(x),0):H_m(x,u_m(x),0)=0\}
\]
and that each point in $\mathcal{K}_{H_m}$ is a fixed point of $\Phi^{H_m}_t$. 
So,  any convex combination of atomic measures supported in $\mathcal{K}_{H_m}$ is a Mather measure for $H_m$. We use $\mathfrak{M}_{m}$ to denote the set of all convex combinations of atomic measures supported in $\mathcal{K}_{H_m}$.

Define the set-valued map 
	 \[
	 \Psi: \mathcal{P}(M) \rightrightarrows \mathcal{P}(M),\quad m \mapsto \Psi(m),
	 \] 
	 where 
	 \[
	 \Psi(m):=\left\{\pi_x \sharp \eta_{m}:\ \eta_{m} \in \mathfrak{M}_{m} \right\}.
	 \]
In view of the arguments in the Introduction, it is important to show that there exists a fixed point $\bar{m}$ of $\Psi$.

Note that the metric space $(\mathcal{P}(M),d_1)$ is convex and compact due to 
	 Prokhorov's theorem (see, for instance, \cite{bib:BB}).  
	 Since $\Psi$ has nonempty convex values, the only hypothesis of
	 Kakutani's theorem we need to check is that $\Psi$ has closed graph: for any pair of sequences $\{m_j\}_{j\in\mathbb{N}}\subset\mathcal{P}(M)$,  $\{\mu_{j}\}_{j\in\mathbb{N}}\subset\mathcal{P}(M)$ such that
	\[
	  m_j \stackrel{w^*}{\longrightarrow}m,\quad  \mu_{j} \stackrel{w^*}{\longrightarrow}\mu,\ \text{as}\ j\to+\infty \quad \text{and} \quad \mu_j\in \Psi(m_j) \quad \text{for all j} \in\mathbb{N},
	  \]
	  we aim to prove that $\mu \in\Psi(m)$.

Since $\mu_j\in \Psi(m_j)$, there are  measures $\eta_{m_j}\in \mathfrak{M}_{m_j}$ such that $\mu_j=\pi\sharp \eta_{m_j}$. From \eqref{b7}, we have that 
\[
\|u_{m_j}\|_\infty,\ \|u_{m}\|_\infty \leq D_3,\quad \forall j\in\mathbb{N}.
\]
By Proposition \ref{key2}, we get that
\[
\supp(\eta_{m_j})\subset M\times[-D_3,D_3]\times\{0\}=:K_0,\quad \forall j\in\mathbb{N}.
\]
Thus, the sequence $\{\eta_{m_j}\}_{j\in\mathbb{N}}$ is tight.
By Prokhorov's theorem again, passing to a subsequence if necessary, we may suppose that  
	 \begin{align*}
	  \eta_{m_j} \stackrel{w^*}{\longrightarrow}\eta,\  \text{as}\ j\to+\infty \quad \text{and} \quad \mu=\pi\sharp \eta,
	 \end{align*}
where $\eta\in \mathcal{P}_1(T^*M)$.
It suffices to show that $\eta\in\mathfrak{M}_{m}$. 
We first show that $\eta$ is a $\Phi^{H_m}_t$-invariant measure.
Since $\eta_{m_j}$ are $\Phi^{H_{m_j}}_t$-invariant measures, then we deduce that, for any given $t\in \R$,
\begin{align}\label{inva}
\int_{K_0}f(\Phi^{H_{m_j}}_t(x,u,p))\ d\eta_{m_j}=\int_{K_0}f(x,u,p)\ d\eta_{m_j}, \quad \forall f\in C(K_0),\ \forall j\in\mathbb{N}.
\end{align}
Note that $H$ is of class $C^3$, $F$ satisfies (F1) and (F2).
Since $K_0$ is compact, then by the continuous dependence of the solutions on the initial condition and a parameter, we get that
\[
\lim_{j\to\infty}f(\Phi^{H_{m_j}}_t(x,u,p))=f(\Phi^{H_{m}}_t(x,u,p))
\]
uniformly on $K_0$. Thus, by \eqref{inva}, we deduce that
\begin{align*}
\int_{K_0}f(\Phi^{H_{m}}_t(x,u,p))\ d\eta=\int_{K_0}f(x,u,p)\ d\eta, \quad \forall f\in C(K_0),	
\end{align*}
which shows that $\eta$ is  $\Phi^{H_m}_t$-invariant. Next, we show that $\supp(\eta)\subset \mathcal{K}_{H_m}=\{(x,u_m(x),0):H_m(x,u_m(x),0)=0\}.
$
Since $\eta_{m_j}\stackrel{w^*}{\longrightarrow}\eta$ as $j\to\infty$, for any $(x_0,u_0,p_0)\in \supp(\eta)$, there is a sequence of points $(x_j,u_j,p_j)\in\supp(\eta_{m_j})$ with $(x_j,u_j,p_j)\to (x_0,u_0,p_0)$ as $j\to\infty$. By Proposition \ref{key2} and $\eta_{m_j}\in\mathfrak{M}_{m_j}$, we deduce that
$u_j=u_{m_j}(x_j)$, $p_j=0$, and that
\[
H(x_j,u_{m_j}(x_j),0)-F(x_j,m_j)=0
\]
for all $j\in\mathbb{N}$. By Proposition \ref{key3}, the equi-Lipschitz property of $\{u_{m_j}\}$ and (F1), we get that
\[
 H(x_0,u_{m}(x_0),0)-F(x_0,m)=0,
\]
which shows that $\supp(\eta)\subset \mathcal{K}_{H_m}$. Thus, $\eta\in\mathfrak{M}_{m}$.	 
	 So far, we have proved that $\Psi$ has closed graph.  
	 By Kakutani's  theorem, there exists $\bar m\in \mathcal{P}(M)$ such that $\bar m\in\Psi(\bar m)$.

Denote by  $\bar{u}$ the unique viscosity solution of $u+H_{\bar{m}}(x,Du)=0$. From the arguments in Section 2, $\bar u$ is differentiable $\bar m$ -a.e since $\bar m$ is supported on a subset of the projected Mather set $\mathcal{M}_{H_{\bar m}}$.

	 For any $x\in \supp(\bar m)$, let $\gamma_{t}(x)=\pi_x\circ \Phi_{t}^{H_{\bar{m}}}(x,\bar{u}(x),D\bar u(x))$.
	 Then, we have that 
	 $$\frac{d}{dt}\gamma_{t}(x)=\frac{\partial H_{\bar m}}{\partial p}\left(\gamma_{t}(x),\bar{u}(\g_t(x)), D\bar u(\gamma_{t}(x))\right).$$ 
	 Since the map $\pi_x: \supp(\eta_{\bar m}) \to \supp(\bar m)$ is one-to-one and its inverse is given by $x \mapsto (x, \bar{u}(x), D\bar u(x)))$ on $\supp(\bar m)$, then $\gamma_t:\supp(\bar m)\to \supp(\bar m)$ is a bijection for each $t\in\R$. Note that, for each $t\in\R$ and any function $f \in C^{1}(M)$, we get that
	 \begin{align*}\label{4-509}
	 	\begin{split}
	 \int_{\supp(\bar m)}f(\gamma_t(x))d\bar m&=\int_{\supp(\bar m)}f\circ \gamma_t(x)d\pi_x\sharp\eta_{\bar m}\\&=\int_{\supp(\eta_{\bar m})}f\circ \gamma_t(\pi_x(x,u,p))d\eta_{\bar m}\\
	 &=\int_{\supp(\eta_{\bar m})}f (\pi_x\circ \Phi^{H_{\bar m}}_t(x,u,p))d\eta_{\bar m}\\
	 &=\int_{\supp(\eta_{\bar m})}f (\pi_x(x,u,p))d\eta_{\bar m}\\
	 &=\int_{\supp(\bar m)}f (x)d\bar m.
	 \end{split}
	 \end{align*}
	 Here, the first equality holds  since $\bar m$ is a fixed point of $\Psi$, the second one holds by the property of the push-forward, the third holds since $\gamma_t$ is a bijection, the fourth one comes from the $\Phi^{H_{\bar m}}_t$-invariance property of $\eta_{\bar m}$, and the last one is again due to the property of the push-forward.
	 So, for any function $f \in C^{1}(M)$, one can deduce that
	 \begin{align*} 
	 0 &=\frac{d}{dt}\int_{M}{f(\gamma_{t}(x))\ d\bar m(x)}= \int_{M}{\big\langle Df(\gamma_{t}(x)), \frac{\partial H_{\bar m}}{\partial p}(\gamma_{t}(x),\bar{u}(\g_t(x)), D\bar u(\gamma_{t}(x))) \big\rangle_x\ d\bar m(x)} \\&= \int_{M}{\big\langle Df(x), \frac{\partial H_{\bar m}}{\partial p}(x,\bar{u}(x), D\bar u(x)) \big\rangle_x\ d\bar m(x)}. 
	 \end{align*} 
	 Hence, $\bar m$ satisfies the continuity equation which  completes the proof.

	  \end{proof}

\medskip
\noindent {\bf Acknowledgements:}
Kaizhi Wang is supported by NSFC Grant No. 11771283, 11931016 and Innovation Program of Shanghai Municipal Education Commission No. 2021-01-07-00-02-E00087.



\end{document}